\documentclass{article}
\usepackage{makeidx}
\usepackage{amsfonts}
\usepackage{amsmath}

\newtheorem{theorem}{Theorem}

\newtheorem{definition}[theorem]{Definition}
\newtheorem{example}[theorem]{Example}

\newtheorem{proposition}[theorem]{Proposition}
\newtheorem{remark}[theorem]{Remark}

\newenvironment{proof}[1][Proof]{\noindent\textbf{#1.} }{\ \rule{0.5em}{0.5em}}

\begin{document}

\title{Dual jet geometrical objects of momenta in the time-dependent
Hamilton geometry}
\author{Mircea Neagu and Alexandru Oan\u{a}}
\date{}
\maketitle

\begin{abstract}
The aim of this paper is to obtain on the dual 1-jet space $J^{1\ast }(%
\mathbb{R},M)$ the main geometrical objects used in the dual jet geo\-metry
of time-dependent Hamiltonians. We talk about distinguished (d-) tensors,
time-dependent semisprays, nonlinear connections and their mathematical
connections.
\end{abstract}



\section{Introduction}

According to Olver's opinion \cite{Olv1}, we recall that the 1-jet spaces
and their duals are the fundamental ambient mathematical spaces used in the
study of classical and quantum field theories in their Lagrangian and
Hamiltonian approaches (see also \cite{Ata-Nea-Oana}). For this reason, the
studies of Miron \cite{Miron-Hamilton} and Atanasiu (\cite{At}, \cite%
{Atan-Klepp}) led to the development of the \textit{Hamilton geometry of
cotangent bundles} exposed by Miron, Hrimiuc, Shimada and Sab\u{a}u in the
monograph \cite{Miron-Hr-Shim-Sab}. We emphasize that, via the Legendre
duality of the Hamilton spaces with the Lagrange spaces, the preceding
authors have shown in \cite{Miron-Hr-Shim-Sab} that the theory of Hamilton
spaces has the same symmetry as the Lagrange geometry, giving thus a
geometrical framework for the Hamiltonian theory of Analytical Mechanics.

According to this physical and geometrical context, suggested by the
cotangent bundle framework of the Miron et al., this paper is devoted to
exposing a particular case of the \textit{time-dependent covariant Hamilton
geometry} studied in \cite{Ata-Nea-Oana} \textit{on dual 1-jet spaces} (in
the sense of d-tensors, time-dependent semisprays of momenta and nonlinear
connections), which is a natural dual jet extension of the Hamilton geometry
on the cotangent bundle from \cite{Miron-Hr-Shim-Sab}.

\section{The dual 1-jet space}

In our geometrical study we start with a smooth real manifold $M^{n}$ of
dimension $n$, whose local coordinates are $(x^{i})_{i=\overline{1,n}}$. Let
us also consider the dual $1$-jet vector bundle (i.e., \textit{the
time-dependent phase space of momenta}) 
\begin{equation*}
J^{1\ast }(\mathbb{R},M)\equiv \mathbb{R}\times T^{\ast }M\rightarrow 
\mathbb{R}\times M, 
\end{equation*}
whose local coordinates are denoted by $(t,x^{i},p_{i}^{1})$, where the
coordinates $p_{i}^{1}$ have the physical meaning of \textit{momenta}.

The coordinate transformations $(t,x^{i},p_{i}^{1})\longleftrightarrow (%
\tilde{t},\tilde{x}^{i},\tilde{p}_{i}^{1})$ induced from $\mathbb{R}\times M 
$ on the dual $1$-jet space $J^{1\ast }(\mathbb{R},M)$ are given by 
\begin{equation}
\left\{ 
\begin{array}{l}
\tilde{t}=\tilde{t}\left( t\right) \medskip \\ 
\tilde{x}^{i}=\tilde{x}^{i}\left( x^{j}\right) \medskip \\ 
\tilde{p}_{i}^{1}=\dfrac{\partial x^{j}}{\partial \tilde{x}^{i}}\dfrac{d%
\tilde{t}}{dt}p_{j}^{1},%
\end{array}%
\right.  \label{schp}
\end{equation}%
where $d\tilde{t}/dt\neq 0$ and $\det (\partial \tilde{x}^{i}/\partial
x^{j})\neq 0$. It follows that, in our dual jet geo\-me\-tri\-cal approach,
we use a \textit{"relativistic"} time $t$.

By comparison, in the cotangent Hamiltonian approach from \cite%
{Miron-Hr-Shim-Sab}, the authors use the trivial bundle $\mathbb{R}\times
T^{\ast }M\rightarrow T^{\ast }M$, whose coordinates are $\left(
t,x^{i},p_{i}\right) $. In this context, the changes of coordinates are
given by 
\begin{equation*}
\left\{ 
\begin{array}{l}
\tilde{t}=t\medskip \\ 
\tilde{x}^{i}=\tilde{x}^{i}\left( x^{j}\right) \medskip \\ 
\tilde{p}_{i}=\dfrac{\partial x^{j}}{\partial \tilde{x}^{i}}p_{j},%
\end{array}%
\right.
\end{equation*}%
emphasizing the \textit{absolute} character of the time $t$. In such a
context, a time dependent Hamiltonian is a real valued function $H$ on $%
\mathbb{R}\times T^{\ast }M$, which is also called \textit{rheonomic}, or 
\textit{non-autonomous} Hamiltonian. A geometrization of these Hamiltonians
was realized by Miron, Atanasiu and their co-workers in the works \cite{At}, 
\cite{Atan-Klepp}, \cite{Miron-Hamilton} and \cite{Miron-Hr-Shim-Sab}.

Now, doing a transformation of coordinates (\ref{schp}) on $J^{1\ast }(%
\mathbb{R},M)$, we obtain the following results:

\begin{proposition}
The elements of the local natural basis of vector fields%
\begin{equation*}
{\left\{ {\dfrac{\partial }{\partial t}},{\dfrac{\partial }{\partial x^{i}}},%
{\dfrac{\partial }{\partial p_{i}^{1}}}\right\} \subset \mathcal{X}(J^{1\ast
}(\mathbb{R},M))}
\end{equation*}%
transform by the rules%
\begin{equation}
\begin{array}{l}
{{\dfrac{\partial }{\partial t}}={\dfrac{d\tilde{t}}{dt}}{\dfrac{\partial }{%
\partial \tilde{t}}}+{\dfrac{\partial \tilde{p}_{j}^{1}}{\partial t}}{\dfrac{%
\partial }{\partial \tilde{p}_{j}^{1}}}},\medskip \\ 
{{\dfrac{\partial }{\partial x^{i}}}={\dfrac{\partial \tilde{x}^{j}}{%
\partial x^{i}}}{\dfrac{\partial }{\partial \tilde{x}^{j}}}+{\dfrac{\partial 
\tilde{p}_{j}^{1}}{\partial x^{i}}}{\dfrac{\partial }{\partial \tilde{p}%
_{j}^{1}}}},\medskip \\ 
{{\dfrac{\partial }{\partial p_{i}^{1}}}={\dfrac{\partial x^{i}}{\partial 
\tilde{x}^{j}}}{\dfrac{d\tilde{t}}{dt}}{\dfrac{\partial }{\partial \tilde{p}%
_{j}^{1}}};}%
\end{array}
\label{schch}
\end{equation}
\end{proposition}

\begin{proposition}
The elements of the local natural basis of covector fields%
\begin{equation*}
\{dt,dx^{i},dp_{i}^{1}\}\subset \mathcal{X}^{\ast }({J^{1\ast }(\mathbb{R},M)%
})
\end{equation*}%
transform by the rules%
\begin{equation}
\begin{array}{l}
{dt={\dfrac{dt}{d\tilde{t}}}d\tilde{t}},\medskip \\ 
{dx^{i}={\dfrac{\partial x^{i}}{\partial \tilde{x}^{j}}}d\tilde{x}^{j}}%
,\medskip \\ 
{dp_{i}^{1}={\dfrac{\partial p_{i}^{1}}{\partial \tilde{t}}}d\tilde{t}+{%
\dfrac{\partial p_{i}^{1}}{\partial \tilde{x}^{j}}}d\tilde{x}^{j}+{\dfrac{%
\partial \tilde{x}^{j}}{\partial x^{i}}}{\dfrac{dt}{d\tilde{t}}}d\tilde{p}%
_{j}^{1}.}%
\end{array}
\label{schfh}
\end{equation}
\end{proposition}

\section{Time-dependent semisprays of momenta}

As in the book \cite{Miron-Hr-Shim-Sab}, a central role in our dual jet
geometrical study is played by \textit{d-tensors}.

\begin{definition}
A geometrical object $T=\left( T_{1j(1)(l)\ldots }^{1i(k)(1)\ldots
}(t,x^{r},p_{r}^{1})\right) $ on the dual $1$-jet space ${J^{1\ast }(\mathbb{%
R},M)}$, whose local components change according to the rules%
\begin{equation*}
T_{1j(1)(l)\ldots }^{1i(k)(1)\ldots }=\tilde{T}_{1q(1)(s)\ldots
}^{1p(r)(1)\ldots }{\frac{dt}{d\tilde{t}}}{\frac{\partial x^{i}}{\partial 
\tilde{x}^{p}}}\left( {\frac{\partial x^{k}}{\partial \tilde{x}^{r}}}{\frac{d%
\tilde{t}}{dt}}\right) {\frac{d\tilde{t}}{dt}}{\frac{\partial \tilde{x}^{q}}{%
\partial x^{j}}}\left( {\frac{\partial \tilde{x}^{s}}{\partial x^{l}}}{\frac{%
dt}{d\tilde{t}}}\right) \ldots \;
\end{equation*}%
with respect to a transformation of coordinates (\ref{schp}) on ${J^{1\ast }(%
\mathbb{R},M)}$, is called a \textbf{d-tensor} or a \textbf{distinguished
tensor field} on ${J^{1\ast }(\mathbb{R},M)}$.
\end{definition}

\begin{remark}
The placing between parentheses of certain indices of the local components $%
T_{1j(1)(l)\ldots }^{1i(k)(1)\ldots }$ is necessary for clearer future
contractions.
\end{remark}

\begin{example}
If $H:{J^{1\ast }(\mathbb{R},M)}\rightarrow \mathbb{R}$ is a Hamiltonian
function depending on the momenta $p_{i}^{1}$ then the local components%
\begin{equation*}
G_{(1)(1)}^{(i)(j)}={\frac{1}{2}}{\frac{\partial ^{2}H}{\partial
p_{i}^{1}\partial p_{j}^{1}}}
\end{equation*}%
represent a d-tensor field $\mathbb{G}=\left( G_{(1)(1)}^{(i)(j)}\right) $
which is called the \textbf{vertical fundamental\ metrical d-tensor}
produced by $H.$
\end{example}

\begin{example}
The distinguished tensor $\mathbb{C}=\left( \mathbb{C}_{(i)}^{(1)}\right) ,$%
where $\mathbb{C}_{(i)}^{(1)}=p_{i}^{1},$ is called the \textbf{%
Liou\-ville-Ha\-mil\-ton d-tensor field of momenta} on the dual $1$-jet
space ${J^{1\ast }(\mathbb{R},M)}$.
\end{example}

\begin{example}
If $h_{11}(t)$ is a semi-Riemannian metric on $\mathbb{R}$, then the
geometrical object $\mathbb{L}=\left( L_{(j)11}^{(1)}\right) ,$ where $%
L_{(j)11}^{(1)}=h_{11}p_{j}^{1},$ is called the \textbf{momentum
Liou\-ville-Ha\-mil\-ton d-tensor associated with the metric }$h_{11}(t)$.
\end{example}

\begin{example}
Using the preceding metric $h_{11}(t)$, the distinguished tensor $\mathbb{J}%
=\left( J_{(1)1j}^{(i)}\right),$ where $J_{(1)1j}^{(i)}=h_{11}\delta _{j}^{i}
$, is called the \textbf{d-tensor of }$h$\textbf{-normalization } on the
dual 1-jet space ${J^{1\ast }(\mathbb{R},M)}$.
\end{example}

It is obvious that any d-tensor on ${J^{1\ast }(\mathbb{R},M)}$ is a tensor
field on ${J^{1\ast }(\mathbb{R},M)}$. Conversely, the opposite is not true.
As examples, we construct two tensors on ${J^{1\ast }(\mathbb{R},M)}$, which
are not d-tensors on ${J^{1\ast }(\mathbb{R},M)}$.

\begin{definition}
A global tensor $\underset{1}{G}$ on ${J^{1\ast }(\mathbb{R},M)},$ locally
expressed by%
\begin{equation*}
\underset{1}{G}=p_{i}^{1}dx^{i}\otimes {\frac{\partial }{\partial t}}-2%
\underset{1}{G}\text{{}}_{(j)i}^{(1)}dx^{i}\otimes {\frac{\partial }{%
\partial p_{j}^{1}}},
\end{equation*}%
is called a \textbf{temporal semispray} on the dual $1$-jet space ${J^{1\ast
}(\mathbb{R},M)}$.
\end{definition}

Taking into account that the temporal semispray $\underset{1}{G}$ is a
global tensor on ${J^{1\ast }(\mathbb{R},M)}$, by a direct calculation, we
obtain

\begin{proposition}
\emph{(i)} Under a transformation of coordinates (\ref{schp}) the local
components $\underset{1}{G}${}$_{(j)i}^{(1)}$ of the global tensor $\underset%
{1}{G}$ change according to the rules%
\begin{equation}
2\underset{1}{\widetilde{G}}\text{{}}_{(k)r}^{(1)}=2\underset{1}{G}\text{{}}%
_{(j)i}^{(1)}{\frac{d\tilde{t}}{dt}}{\frac{\partial x^{i}}{\partial \tilde{x}%
^{r}}}{\frac{\partial x^{j}}{\partial \tilde{x}^{k}}}-{\frac{\partial x^{i}}{%
\partial \tilde{x}^{r}}}{\frac{\partial \tilde{p}_{k}^{1}}{\partial t}}%
p_{i}^{1}.  \label{tsprh}
\end{equation}

\emph{(ii)} Conversely, to give a temporal semispray on ${J^{1\ast }(\mathbb{%
R},M)}$ is equivalent to give a set of local functions $\underset{1}{G}%
=\left( \underset{1}{G}{}_{(j)i}^{(1)}\right) $ which transform by the rules
(\ref{tsprh}).
\end{proposition}

\begin{example}
If $H_{11}^{1}(t)=(h^{11}/2)(dh_{11}/dt)$ is the Christoffel symbol of a
semi-Riemannian me\-tric $h_{11}(t)$ of the temporal manifold $\mathbb{R}$,
then the local components 
\begin{equation}
\underset{1}{\overset{0}{G}}\text{{}}_{(j)k}^{(1)}=\frac{1}{2}%
H_{11}^{1}p_{j}^{1}p_{k}^{1}  \label{temp-semispray-assoc-metric}
\end{equation}%
represent a temporal semispray $\underset{1}{\overset{0}{G}}$ on the dual $1$%
-jet space ${J^{1\ast }(\mathbb{R},M)}$, which is called the \textbf{%
canonical temporal semispray associated with the metric }$h_{11}(t)$.
\end{example}

A second example of tensor on the dual 1-jet space ${J^{1\ast }(\mathbb{R},M)%
},$ which is not a distinguished tensor, is given by

\begin{definition}
A global tensor $\underset{2}{G}$ on ${J^{1\ast }(\mathbb{R},M)},$ locally
expressed by%
\begin{equation*}
\underset{2}{G}=\delta _{i}^{j}dx^{i}\otimes {\frac{\partial }{\partial x^{j}%
}}-2\underset{2}{G}{}_{(j)i}^{(1)}dx^{i}\otimes {\frac{\partial }{\partial
p_{j}^{1}},}
\end{equation*}%
is called a \textbf{spatial semispray} on the dual $1$-jet space ${J^{1\ast
}(\mathbb{R},M)}$.
\end{definition}

Like in the case of a temporal semispray, we can prove without difficulties
the following statement:

\begin{proposition}
To give a spatial semispray on ${J^{1\ast }(\mathbb{R},M)}$ is equivalent to
give a set of local functions $\underset{2}{G}=\left( \underset{2}{G}\text{{}%
}_{(j)i}^{(1)}\right) $ which transform by the rules%
\begin{equation}
2\underset{2}{\widetilde{G}}\text{{}}_{(s)k}^{(1)}=2\underset{2}{G}\text{{}}%
_{(j)i}^{(1)}{\frac{d\tilde{t}}{dt}}{\frac{\partial x^{i}}{\partial \tilde{x}%
^{k}}}{\frac{\partial x^{j}}{\partial \tilde{x}^{s}}}-{\frac{\partial x^{i}}{%
\partial \tilde{x}^{k}}}{\frac{\partial \tilde{p}_{s}^{1}}{\partial x^{i}}}.
\label{ssprh}
\end{equation}
\end{proposition}

\begin{example}
If $\gamma _{jk}^{i}(x)$ are the Christoffel symbols of a semi-Riemannian
me\-tric $\varphi _{ij}(x)$ of the spatial manifold $M$, then the local
components 
\begin{equation}
\underset{2}{\overset{0}{G}}\text{{}}_{(j)k}^{(1)}=-\frac{1}{2}\gamma
_{jk}^{i}p_{i}^{1}  \label{spatial-semispr-assoc-metric}
\end{equation}%
define a spatial semispray $\underset{2}{\overset{0}{G}}$ on the dual $1$%
-jet space ${J^{1\ast }(\mathbb{R},M)}$, which is called the \textbf{%
canonical spatial semispray associated with the metric }$\varphi _{ij}(x)$.
\end{example}

\begin{definition}
A pair $G=\left( \underset{1}{G},\underset{2}{G}\right) $, consisting of a
temporal semispray $\underset{1}{G}$ and a spatial semispray $\underset{2}{G}
$, is called a \textbf{time-dependent semispray of momenta} on the dual $1$%
-jet space ${J^{1\ast }(\mathbb{R},M)}$.
\end{definition}

\section{Nonlinear connections and adapted bases}

In what follows, we study the important geometrical concept of \textit{%
nonlinear connection} on the dual $1$-jet space ${J^{1\ast }(\mathbb{R},M)}$%
, which is intimately related by the concept of time-dependent semispray.

\begin{definition}
A pair of local functions $N=\left( \underset{1}{N}\text{{}}_{(k)1}^{(1)},%
\underset{2}{N}\text{{}}_{(k)i}^{(1)}\right) $ on ${J^{1\ast }(\mathbb{R},M)}%
,$ which transform by the rules%
\begin{equation}
\begin{array}{l}
\underset{1}{\widetilde{N}}\text{{}}{_{(j)1}^{(1)}=}\underset{1}{N}\text{{}}%
_{(k)1}^{(1)}{{\dfrac{\partial x^{k}}{\partial \tilde{x}^{j}}}-\dfrac{dt}{d%
\tilde{t}}{\dfrac{\partial \tilde{p}_{j}^{1}}{\partial t}}},\medskip \\ 
\underset{2}{\widetilde{N}}\text{{}}{_{(j)r}^{(1)}=\underset{2}{N}%
{}_{(k)i}^{(1)}{\dfrac{d\tilde{t}}{dt}}{\dfrac{\partial x^{k}}{\partial 
\tilde{x}^{j}}}}\dfrac{\partial x^{i}}{\partial \tilde{x}^{r}}{-\dfrac{%
\partial x^{i}}{\partial \tilde{x}^{r}}{\dfrac{\partial \tilde{p}_{j}^{1}}{%
\partial x^{i}}}},%
\end{array}
\label{schcoh}
\end{equation}%
is called a \textbf{nonlinear connection }on the dual $1$-jet bundle ${%
J^{1\ast }(\mathbb{R},M)}$. The geo\-me\-tri\-cal entity$\ \underset{1}{N}%
=\left( \underset{1}{N}\underset{}{\overset{\left( 1\right) }{_{\left(
j\right) 1}}}\right) \ $(respectively $\underset{2}{N}=\left( \underset{2}{N}%
\underset{}{\overset{\left( 1\right) }{_{\left( j\right) i}}}\right) $) is
called a \textbf{temporal} (respectively \textbf{spatial}) \textbf{nonlinear
connection }on ${J^{1\ast }(\mathbb{R},M)}.$
\end{definition}

Now, let us expose the connection between the time-dependent semisprays of
momenta and nonlinear connections on the dual 1-jet space ${J^{1\ast }(%
\mathbb{R},M)}$. For this, let us consider that $\varphi _{ij}(x)$ is a
semi-Riemannian metric on the spatial manifold $M$. Thus, using the
transformation rules (\ref{tsprh}), (\ref{ssprh}) and (\ref{schcoh}) of the
geometrical objects taken in study, we can easily prove the following
statements:

\begin{proposition}
\emph{(i)} The connection between the temporal semisprays $\underset{1}{G}%
=\left( \underset{1}{G}{}_{(j)k}^{(1)}\right) $ and the temporal components
of nonlinear connections $N_{\text{temporal}}=\left( \underset{1}{N}\text{{}}%
_{(r)1}^{(1)}\right) $ is given by the relations%
\begin{equation*}
\underset{1}{N}\text{{}}_{(r)1}^{(1)}=\varphi ^{jk}{\frac{\partial \underset{%
1}{G}{}_{(j)k}^{(1)}}{\partial p_{i}^{1}}}\varphi _{ir},\qquad\underset{1}{G}%
\text{{}}_{(i)j}^{(1)}=\frac{1}{2}\underset{1}{N}{}_{(i)1}^{(1)}p_{j}^{1}.
\end{equation*}

\emph{(ii)} The connection between spatial semisprays $\underset{2}{G}%
=\left( \underset{2}{G}\text{{}}_{(j)i}^{(1)}\right) $ and the spatial
components of nonlinear connections $N_{\text{spatial}}=\left( \underset{2}{N%
}\text{{}}_{(j)i}^{(1)}\right) $ is given via the relations%
\begin{equation*}
\underset{2}{N}\text{{}}_{(j)i}^{(1)}=2\underset{2}{G}\text{{}}%
_{(j)i}^{(1)},\qquad\underset{2}{G}\text{{}}_{(j)i}^{(1)}{={\frac{1}{2}}}%
\underset{2}{N}\text{{}}_{(j)i}^{(1)}.
\end{equation*}
\end{proposition}

\begin{remark}
It is obvious that on the 1-jet space ${J^{1\ast }( \mathbb{R},M)}$ a
time-dependent semispray of momenta $G$ naturally induces a nonlinear
connection $N_{G}$ and vice-versa, a nonlinear connection $N$ induces a
time-dependent semispray $G_{N}$. The nonlinear connection $N_{G}$ is called
the \textbf{canonical nonlinear connection associated with the
time-dependent semispray of momenta} $G$ and \textbf{vice-versa}.
\end{remark}

\begin{example}
The canonical nonlinear connection $\overset{0}{N}=\left( \underset{1}{%
\overset{0}{N}}\text{{}}_{(i)1}^{(1)},\underset{2}{\overset{0}{N}}\text{{}}%
_{(i)j}^{(1)}\right) $ produced by the canonical time-dependent semispray of
momenta $\overset{0}{G}=\left( \underset{1}{\overset{0}{G}},\underset{2}{%
\overset{0}{G}}\right) $ has the local components%
\begin{equation}
\underset{1}{\overset{0}{N}}\text{{}}_{(i)1}^{(1)}=H_{11}^{1}p_{i}^{1},%
\qquad \underset{2}{\overset{0}{N}}\text{{}}_{(i)j}^{(1)}=-\gamma
_{ij}^{k}p_{k}^{1}.  \label{nlc-assoc-to-metric}
\end{equation}%
This nonlinear connection is called the \textbf{canonical nonlinear
connection on }${J^{1\ast }(\mathbb{R},M)}$\textbf{, associated with the
semi-Riemannian metrics }$h_{11}(t)$\textbf{\ and }$\varphi _{ij}(x).$
\end{example}

Taking into account the complicated transformation rules (\ref{schch}) and (%
\ref{schfh}), we need a \textit{horizontal distribution }on the dual 1-jet
space ${J^{1\ast }(\mathbb{R},M)}$, in order to construct some \textit{%
adapted bases of vector }and\textit{\ covector fields}, whose transformation
rules are simpler (tensorial ones, for instance).

In this direction, let $u^{\ast }=\left( t,x^{i},p_{i}^{1}\right) \in {%
J^{1\ast }(\mathbb{R},M)}$ be an arbitrary point and let us consider the
differential map%
\begin{equation*}
\pi ^{\ast }\ _{\ast ,u^{\ast }}:T_{u^{\ast }}{J^{1\ast }(\mathbb{R},M)}%
\rightarrow T_{\left( t,x\right) }\left( \mathbb{R}\times M\right)
\end{equation*}%
of the canonical projection%
\begin{equation*}
\pi ^{\ast }:{J^{1\ast }(\mathbb{R},M)}\rightarrow \mathbb{R}\times M,\text{
\ \ \ }\pi ^{\ast }\left( u^{\ast }\right) =\left( t,x\right) ,
\end{equation*}%
together with its vector subspace $W_{u^{\ast }}=Ker\pi ^{\ast }\ _{\ast
,u^{\ast }}\subset T_{u^{\ast }}{J^{1\ast }(\mathbb{R},M)}.$ Because the
differential map $\pi ^{\ast }\ _{\ast ,u^{\ast }}$ is a surjection, we find
that we have $\dim _{\mathbb{R}}W_{u^{\ast }}=n$ and, moreover, a basis in $%
W_{u^{\ast }\text{ }}$is determined by $\left\{ \left. \dfrac{\partial }{%
\partial p_{i}^{1}}\right\vert _{u^{\ast }}\right\} .$

So, the map $\mathcal{W}:u^{\ast }\in {J^{1\ast }(\mathbb{R},M)}\rightarrow
W_{u^{\ast }}\subset T_{u^{\ast }}{J^{1\ast }(\mathbb{R},M)}$ is a
differential distribution, which is called the \textit{vertical distribution}
on the dual 1-jet space ${J^{1\ast }(\mathbb{R},M)}.$

\begin{definition}
A differential distribution%
\begin{equation*}
\mathcal{H}:u^{\ast }\in {J^{1\ast }(\mathbb{R},M)}\rightarrow H_{u^{\ast
}}\subset T_{u^{\ast }}{J^{1\ast }(\mathbb{R},M)},
\end{equation*}%
which is supplementary to the vertical distribution $\mathcal{W},$ that is
we have%
\begin{equation*}
T_{u^{\ast }}{J^{1\ast }(\mathbb{R},M)}=H_{u^{\ast }}\oplus W_{u^{\ast },}%
\text{ }\forall \text{ }u^{\ast }\in {J^{1\ast }(\mathbb{R},M)},
\end{equation*}%
is called a \textbf{horizontal distribution} on the dual $1$-jet space ${%
J^{1\ast }(\mathbb{R},M)}$.
\end{definition}

The above definition implies that $\dim _{\mathbb{R}}H_{u^{\ast }}=n+1,$ $%
\forall $ $u^{\ast }\in {J^{1\ast }(\mathbb{R},M)}.$ Moreover, the Lie
algebra of the vector fields $\mathcal{X}\left( {J^{1\ast }(\mathbb{R},M)}%
\right) $ can be decomposed in the direct sum $\ \mathcal{X}\left( {J^{1\ast
}(\mathbb{R},M)}\right) =\mathcal{S}\left( \mathcal{H}\right) \oplus 
\mathcal{S}\left( \mathcal{W}\right) ,$ where $\mathcal{S}\left( \mathcal{H}%
\right) $ (respectively $\mathcal{S}\left( \mathcal{W}\right) $) is the set
of differentiable sections on $\mathcal{H}$ (respectively $\mathcal{W}$)$.$

Supposing that $\mathcal{H}$\ is a fixed horizontal distribution on ${%
J^{1\ast }(\mathbb{R},M)}$, we have the isomorphism%
\begin{equation*}
\left. \pi ^{\ast }\ _{\ast ,u^{\ast }}\right\vert _{H_{u^{\ast
}}}:H_{u^{\ast }}\rightarrow T_{\pi ^{\ast }\left( u^{\ast }\right) }\left( 
\mathbb{R}\times M\right) ,
\end{equation*}%
which allows us to prove the following result:

\begin{theorem}
\emph{(i)} There exist unique linear independent horizontal vector fields $%
\dfrac{\delta }{\delta t},$ $\dfrac{\delta }{\delta x^{i}}\in \mathcal{S}%
\left( \mathcal{H}\right) ,$ having the properties%
\begin{equation}
\pi ^{\ast }\ _{\ast }\left( \dfrac{\delta }{\delta t}\right) =\dfrac{%
\partial }{\partial t},\quad \pi ^{\ast }\ _{\ast }\left( \dfrac{\delta }{%
\delta x^{i}}\right) =\dfrac{\partial }{\partial x^{i}}.
\label{delta-t-si-x}
\end{equation}

\emph{(ii)} The horizontal vector fields $\dfrac{\delta }{\delta t}$ and $%
\dfrac{\delta }{\delta x^{i}}$ can be uniquely written in the form%
\begin{equation}
\dfrac{\delta }{\delta t}=\dfrac{\partial }{\partial t}-\underset{1}{N}%
\underset{}{\overset{\left( 1\right) }{_{\left( j\right) 1}}}\dfrac{\partial 
}{\partial p_{j}^{1}},\qquad \dfrac{\delta }{\delta x^{i}}=\dfrac{\partial }{%
\partial x^{i}}-\underset{2}{N}\underset{}{\overset{\left( 1\right) }{%
_{\left( j\right) i}}}\dfrac{\partial }{\partial p_{j}^{1}}.
\label{form-of-delta-t-si-x}
\end{equation}

\emph{(iii)} The local coefficients \ $\underset{1}{N}\underset{}{\overset{%
\left( 1\right) }{_{\left( j\right) 1}}}$\ and $\underset{2}{N}\underset{}{%
\overset{\left( 1\right) }{_{\left( j\right) i}}}$\ obey the rules (\ref%
{schcoh}) of a nonlinear connection $N$ on ${J^{1\ast }(\mathbb{R},M)}.$

\emph{(iv)} On the 1-jet space ${J^{1\ast }( \mathbb{R},M)}$ to give a
horizontal distribution $\mathcal{H}$ is equivalent to give a nonlinear
connection $N=\left( \underset{1}{N}\underset{}{\overset{\left( 1\right) }{%
_{\left( j\right) 1}}},\ \underset{2}{N}\underset{}{\overset{\left( 1\right) 
}{_{\left( j\right) i}}}\right) .$
\end{theorem}

\begin{proof}
Let $\dfrac{\delta }{\delta t},$ $\dfrac{\delta }{\delta x^{i}}\in \mathcal{X%
}\left( {J^{1\ast }(\mathbb{R},M)}\right) $ be vector fields on ${J^{1\ast }(%
\mathbb{R},M)}$, locally expressed by%
\begin{equation*}
\begin{array}{l}
\dfrac{\delta }{\delta t}=A_{1}^{1}\dfrac{\partial }{\partial t}+A_{1}^{j}%
\dfrac{\partial }{\partial x^{j}}+A_{\left( j\right) 1}^{\left( 1\right) }%
\dfrac{\partial }{\partial p_{j}^{1}},\medskip \\ 
\dfrac{\delta }{\delta x^{i}}=X_{i}^{1}\dfrac{\partial }{\partial t}%
+X_{i}^{j}\dfrac{\partial }{\partial x^{j}}+X_{\left( j\right) i}^{\left(
1\right) }\dfrac{\partial }{\partial p_{j}^{1}},%
\end{array}%
\end{equation*}%
which verify the relations (\ref{delta-t-si-x}). Then, taking into account
the local expression of the map $\pi ^{\ast }{}_{\ast }$, we get%
\begin{eqnarray*}
A_{1}^{1} &=&1,\ A_{1}^{j}=0,\ A_{\left( j\right) 1}^{\left( 1\right) }=-%
\underset{1}{N}\underset{}{\overset{\left( 1\right) }{_{\left( j\right) 1}}},
\\
X_{i}^{1} &=&0,\ X_{i}^{j}=\delta _{i}^{j},\ X_{\left( j\right) i}^{\left(
1\right) }=-\underset{2}{N}\underset{}{\overset{\left( 1\right) }{_{\left(
j\right) i}}}.
\end{eqnarray*}%
These equalities prove the form (\ref{form-of-delta-t-si-x}) of the vector
fields from Theorem, together with their linear independence. The uniqueness
of the coefficients $\underset{1}{N}\underset{}{\overset{\left( 1\right) }{%
_{\left( j\right) 1}}}\ $and $\underset{2}{N}\underset{}{\overset{\left(
1\right) }{_{\left( j\right) i}}}$ is obvious.

Because the vector fields $\dfrac{\delta }{\delta t}$ and $\dfrac{\delta }{%
\delta x^{i}}$ are globally defined, we deduce that a change of coordinates (%
\ref{schp}) on ${J^{1\ast }(\mathbb{R},M)}$ produces a transformation of the
local coefficients $\underset{1}{N}\underset{}{\overset{\left( 1\right) }{%
_{\left( j\right) 1}}}$ and $\underset{2}{N}\underset{}{\overset{\left(
1\right) }{_{\left( j\right) i}}}$ by the rules (\ref{schcoh}).

Finally, starting with a set of functions $N=\left( \underset{1}{N}\underset{%
}{\overset{\left( 1\right) }{_{\left( j\right) 1}}},\ \underset{2}{N}%
\underset{}{\overset{\left( 1\right) }{_{\left( j\right) i}}}\right) ,$
which satisfy the rules (\ref{schcoh})$,$ we can construct the horizontal
distribution\ $\mathcal{H}$, taking%
\begin{equation*}
H_{u^{\ast }}=Span\left\{ \left. \frac{\delta }{\delta t}\right\vert
_{u^{\ast }},\left. \frac{\delta }{\delta x^{i}}\right\vert _{u^{\ast
}}\right\} .
\end{equation*}%
The decomposition $T_{u^{\ast }}{J^{1\ast }(\mathbb{R},M)}=H_{u^{\ast
}}\oplus W_{u^{\ast }}$ is obvious now.
\end{proof}

\begin{definition}
The set of the linear independent vector fields%
\begin{equation}
\left\{ \frac{\delta }{\delta t},\frac{\delta }{\delta x^{i}},\frac{\partial 
}{\partial p_{i}^{1}}\right\} \subset \mathcal{X}\left( {J^{1\ast }(\mathbb{R%
},M)}\right)  \label{ad-basis-nlc}
\end{equation}%
is called the \textbf{adapted basis of vector fields produced by the
nonlinear connection} $N=\left( \underset{1}{N},\underset{2}{N}\right) .$
\end{definition}

With respect to the coordinate transformations (\ref{schp}), the elements of
the adapted basis (\ref{ad-basis-nlc}) have their transformation laws as
tensorial ones (in contrast with the transformations rules (\ref{schch})):%
\begin{equation*}
\begin{array}{l}
\dfrac{\delta }{\delta t}=\dfrac{d\tilde{t}}{dt}\dfrac{\delta }{\delta 
\tilde{t}},\medskip \\ 
\dfrac{\delta }{\delta x^{i}}=\dfrac{\partial \tilde{x}^{j}}{\partial x^{i}}%
\dfrac{\delta }{\delta \tilde{x}^{j}},\medskip \\ 
\dfrac{\partial }{\partial p_{i}^{1}}=\dfrac{d\tilde{t}}{dt}\dfrac{\partial
x^{i}}{\partial \tilde{x}^{j}}\dfrac{\partial }{\partial \tilde{p}_{j}^{1}}.%
\end{array}%
\end{equation*}

The dual basis (of covector fields) of the adapted basis (\ref{ad-basis-nlc}%
) is given by%
\begin{equation}
\left\{ dt,dx^{i},\delta p_{i}^{1}\right\} \subset \mathcal{X}^{\ast }\left( 
{J^{1\ast }(\mathbb{R},M)}\right)  \label{ad-cobasis-nlc}
\end{equation}%
where%
\begin{equation*}
\delta p_{i}^{1}=dp_{i}^{1}+\underset{1}{N}\overset{(1)}{\underset{}{%
_{\left( i\right) 1}}}dt+\underset{2}{N}\overset{(1)}{\underset{}{_{\left(
i\right) j}}}dx^{j}.
\end{equation*}

\begin{definition}
The dual basis of covector fields (\ref{ad-cobasis-nlc}) is called the 
\textbf{adapted cobasis of covector fields of the nonlinear connection} $%
N=\left( \underset{1}{N},\underset{2}{N}\right) $.
\end{definition}

Moreover, with respect to transformation laws (\ref{schp}), we obtain the
following tensorial transformation rules:%
\begin{equation*}
\begin{array}{l}
dt=\dfrac{dt}{d\tilde{t}}d\tilde{t},\medskip \\ 
dx^{i}=\dfrac{\partial x^{i}}{\partial \tilde{x}^{j}}d\tilde{x}^{j},\medskip
\\ 
\delta p_{i}^{1}=\dfrac{dt}{d\tilde{t}}\dfrac{\partial \tilde{x}^{j}}{%
\partial x^{i}}\delta \tilde{p}_{j}^{1}.%
\end{array}%
\end{equation*}

As a consequence of the preceding assertions, we find the following simple
result:

\begin{proposition}
\emph{(i)} The Lie algebra of vector fields on ${J^{1\ast }(\mathbb{R},M)}$
decomposes in the direct sum $\mathcal{X}\left( {J^{1\ast }(\mathbb{R},M)}%
\right) =\mathcal{X}\left( \mathcal{H}_{\mathbb{R}}\right) \oplus \mathcal{X}%
\left( \mathcal{H}_{M}\right) \oplus \mathcal{X}\left( \mathcal{W}\right) ,$
where%
\begin{equation*}
\mathcal{X}\left( \mathcal{H}_{\mathbb{R}}\right) =Span\left\{ \frac{\delta 
}{\delta t}\right\} ,\text{ }\mathcal{X}\left( \mathcal{H}_{M}\right)
=Span\left\{ \frac{\delta }{\delta x^{i}}\right\} ,\text{ }\mathcal{X}\left( 
\mathcal{W}\right) =Span\left\{ \frac{\partial }{\partial p_{i}^{1}}\right\}
.
\end{equation*}

\emph{(ii)} The Lie algebra of covector fields on ${J^{1\ast }(\mathbb{R},M)}
$ decomposes in the direct sum $\mathcal{X}^{\ast }\left( {J^{1\ast }(%
\mathbb{R},M)}\right) =\mathcal{X}^{\ast }\left( \mathcal{H}_{\mathbb{R}%
}\right) \oplus \mathcal{X}^{\ast }\left( \mathcal{H}_{M}\right) \oplus 
\mathcal{X}^{\ast }\left( \mathcal{W}\right) ,$ where%
\begin{equation*}
\mathcal{X}^{\ast }\left( \mathcal{H}_{\mathbb{R}}\right) =Span\left\{
dt\right\} ,\text{ }\mathcal{X}^{\ast }\left( \mathcal{H}_{M}\right)
=Span\left\{ dx^{i}\right\} ,\text{ }\mathcal{X}^{\ast }\left( \mathcal{W}%
\right) =Span\left\{ \delta p_{i}^{1}\right\} .
\end{equation*}
\end{proposition}

\begin{definition}
The distributions $\mathcal{H}_{\mathbb{R}}$ and $\mathcal{H}_{M}$ are
called the $\mathbb{R}$\textbf{-horizontal distribution} and $M$\textbf{%
-horizontal distribution} on ${J^{1\ast }(\mathbb{R},M)}$.
\end{definition}

\section{Discussion}

The results of this paper represent the basics for a subsequent
geometrization (in the sense of nonlinear connection, canonical d-linear
connection, d-torsions and d-curvatures) on dual jet spaces of the
time-dependent Hamiltonians regarded as real-valued functions on the 1-jet
space ${J^{1\ast }(\mathbb{R},M)}$. This Hamilton geometrization is similar
with that developed on cotangent bundles (\cite{Miron-Hr-Shim-Sab}, \cite%
{Miron-Hamilton}, \cite{At} and \cite{Atan-Klepp}), but is characterized by
a "relativistic" time in the study. In contrast the time-dependent Hamilton
geometrization on cotangent bundles is characterized by an absolute time.

\noindent \textsc{Mircea Neagu} and \textsc{Alexandru Oan\u{a}}\newline
Transilvania University of Bra\c{s}ov\newline
Department of Mathematics and Informatics\newline
Blvd. Iuliu Maniu 50, 500091 Bra\c{s}ov, Romania.

\noindent E-mails: mircea.neagu@unitbv.ro, alexandru.oana@unitbv.ro

\end{document}